\documentclass[12pt]{article}

\usepackage{amsfonts}
\usepackage{amsmath}
\usepackage{amssymb}
\usepackage{amsthm}
\usepackage{amscd}

\theoremstyle{plain} 
\newtheorem{theo}{Theorem}[section]
\newtheorem{prop}[theo]{Proposition}
\newtheorem{lem}[theo]{Lemma}

\newtheorem{coro}[theo]{Corollary}
\theoremstyle{definition}

\newtheorem{rk}[theo]{Remark}

\numberwithin{equation}{section}

\newcommand{\N}{{\mathbb N}}

\newcommand{\R}{{\mathbb R}}

\newcommand{\calE}{{\cal{E}}}

\newcommand{\ccalE}{\check{\cal{E}}}

\newcommand{\calD}{{\cal{D}}}
\newcommand{\ccalD}{\check{\cal{D}}}

\DeclareMathOperator{\ran}{ran}
\DeclareMathOperator{\disc}{disc}
\DeclareMathOperator{\dom}{dom}

\begin{document}
\title{Computations and global properties for traces of Bessel's Dirichlet form}
\author{\normalsize
Ali BenAmor\footnote{corresponding author} \footnote{High institute for transport and logistics. University of Sousse, Tunisia. E-mail: ali.benamor@ipeit.rnu.tn},
Rafed Moussa \footnote{Department of Mathematics, High school of sciences and technology of Hammam Sousse. University of Sousse, Tunisia. E-mail: rafed.moussa@gmail.com}
}
\date{\today}
\maketitle

\begin{abstract}
We compute explicitly traces of the Dirichlet form related to the Bessel process with respect to discrete measures as well as measures of mixed type. Then some global properties of the obtained Dirichlet forms, such as conservativeness, irreducibility and compact embedding for their domains are discussed.
\end{abstract}



\section{Introduction}
The aim of the paper is two-fold: 1) Compute the trace of the Dirichlet form related to Bessel's operator (process) on some subsets of $\R$ via the method developed in \cite{BBST}.\\
2) Analyze some global  properties of the obtained Dirichlet forms. Mainly we shall be concerned with conservativeness property, irreducibility and global properties of elements of the domain of the trace form as well as compact embedding for their domains.

Traces of quadratic forms in the framework of Hilbert spaces can be performed  via the construction made in \cite{BBST}. In particular for discrete sets the trace operator can be interpreted as a discretization of  Bessel's operator. For sets composed from composites of continuum and discrete sets we obtain mixed-type Bessel's operator on graphs. What we shall obtain, for composite sets is an operator commonly named 'Laplacian on quantum graphs'.

The focus on this example is motivated mainly by the following reasons: First to construct a discrete Bessel operator, second to analyze how stable the properties of conservativeness (or stochastic completeness), irreducibility as well as global properties of functions from the domain of the initial form (decay property for example), when passing to the trace form.\\
Let us precise that the Bessel operator on the right half axis $(0,\infty)$ is the prototype of a conservative transient and irreducible diffusion. A natural question arises, whether these global properties are inherited by the trace form (or equivalently the trace operator). Whereas it is known that in a abstract framework transience is inherited by the trace form (see \cite[Lemma 622, p.317]{Fukushima}), there are no definitive answers concerning conservativeness and irreducibility.

In these notes we will determine, among other results, which measures (or supports of measures) preserve the aforementioned global properties and which do not. In particular our analysis shows that conservation property  for the trace form is strongly correlated to topological as well as geometric properties of  the support of the considered  measures. In particular we will prove that if the support of the measure is finite then the obtained form is never conservative! Whereas, if the support consists of a continuum and a discrete a discrete set, we will prove that the chosen metric plays a central  role to decide whether the obtained form is conservative or not. Let us emphasize that in the latter case the obtained Dirichlet form is a Dirichlet form related to a quantum graph.

\section{The basics}
Let us introduce some notations. We denote by $I:=(0,\infty)$ and $AC(I)$ the space of absolutely continuous function on $I$. For each $n\in\N,\ n\geq 1$ let $m$ be the measure defined  on $I$ by
$$
dm=2x^{n-1}dx.
$$
We designate by
$$
\calD_0:=\{u:I\to\R,\ u\in AC(I),\ \int_0^\infty (u'(x))^2 x^{n-1}\,dx<\infty\},
$$
and $\calE$ the Dirichlet form defined in $L^2(I,m)$ by
\begin{eqnarray}
\calD=\calD_0\cap L^2(I,m),\ \calE[u]= \int_0^\infty (u'(x))^2 x^{n-1}\,dx.
\label{BesselForm}
\end{eqnarray}
Let us   consider the differential expression defined by
\begin{equation}
\mathcal{L} := -\frac{1}{2}\big(\frac{d^2}{dx^2} + \frac{2 \nu+1}{ x} \frac{d}{dx}\big),\ \quad\text{ where } \nu:=\frac{n}{2}-1.
\end{equation}
It is well known that $\calE$ is a regular strongly local Dirichlet form in $L^2(I,m)$ (a diffusion). Moreover, the positive selfadjoint operator  associated with the form $\calE$ via Kato's representation theorem, which we denote by $L$ is defined by
\begin{align}
D(L)&=\{u\in\calD,\ u' \in AC(I),\ \lim_{x\downarrow 0}x^{n-1}u'(x)=0, \mathcal{L} u:= -\frac{1}{2}u'' -\frac{2 \nu +1}{2x}u'\in L^2(I,m)\}\nonumber\\
Lu &= \mathcal{L} u,\ u\in D(L).
\end{align}
It is nothing else but the Bessel operator. In the probabilistic jargon, $L$ is the generator of the Bessel process on the half-line. It is also the radial component of the Laplacian (or the standard Brownian motion) on $\R^n$.\\
Let $p_t(x,y) $ the corresponding heat kernel. It is well known that (see \cite{Bogus-2016})
\begin{equation}
p_t(x,y) := \frac{1}{2t}(xy)^{-\nu} \exp \big( -\frac{x^2 + y^2}{2 t}\big) I_{\nu}\big(\frac{xy}{t}\big),\ \forall\,x,y,t>0,
\label{HeatKernel}
\end{equation}
where $I_\nu$ is the modified Bessel function given by power series
$$
 I_{\nu}(x)=\sum_{k=0}^{\infty}
 \frac{1}{\Gamma(k+\nu+1)k!}\big(\frac{x}{2}\big)^{2k+\nu}.
$$
We quote that $p_t$ is the fundamental solution of the heat equation $-\frac{\partial u}{\partial t}=Lu$.\\
Let $T_t:=e^{-tL},\ t>0$ be the heat semigroup associated with $L$. Then $T_t$ is the integral operator whose kernel is $p_t$. The form $\calE$ (or the operator  $T_t$) is said to be conservative (or stochastically complete) whenever
\begin{eqnarray}
 T_t 1=1 \quad\text{ for some and hence every } t>0,
\end{eqnarray}
where $T_t$ stands for the $L^\infty$-semigroup induced by the Dirichlet form $\calE$. Analytically, conservativeness  means that the heat amount inside the system is conserved. Whereas probabilistically it means that the process has an infinite life time, whatever its start point is. These physical interpretations are the main motivations to study the conservativeness property of a given Dirichlet form.\\
Let us start  by  giving some inequalities. To our best knowledge these inequalities are new.
\begin{theo}[Inequalities]
\begin{enumerate}
\item Sobolev inequality. For every $n\geq 3$ and every  $2<p\leq \frac{2n}{n-2}$, there is a finite constant $c>0$ such that
\begin{eqnarray}
\big(\int_0^\infty |u(t)|^p\, t^{n-1} dt\big)^{2/p}\leq c\calE[u],\ \forall\,u\in\calD.
\label{Sobolev}
\end{eqnarray}
\item Generalized Strauss inequality. For every $n\geq 3$ and every $1/2\leq \sigma\leq 1$ there is a finite constant $c>0$ such that
\begin{eqnarray}
\sup_{t>0} t^{\frac{n-2\sigma}{2}} |u(t)|\leq c\|u\|_{L^2(I,m)}^{1-\sigma}(\calE[u])^{\frac{\sigma}{2}},\ \forall\,u\in\calD.
\label{Strauss}
\end{eqnarray}
\end{enumerate}
\label{Ineq}
\end{theo}
\begin{proof}
Let $u\in\calD$. Let $v$ be the function defined by  $v(z)=u(|z|),\ z\in\R^n\setminus\{0\}$. Then $v$ is radially symmetric and lies in the Sobolev space $H^1(\R^n)$ (use spherical coordinates). Hence, inequality (\ref{Sobolev}) follows from the classical Sobolev inequality by using spherical coordinates once again.\\
In order to prove (\ref{Strauss}) we make use of \cite[Proposition 1]{Ozawa} for $\sigma=1$ and  \cite[Proposition 3]{Ozawa} for $1/2\leq\sigma<1$ to obtain
$$
\sup_{z\in\R^n\setminus\{0\}}|z|^{\frac{n-2\sigma}{2}} |v(z)|\leq c\|v\|_{L^2(\R^n,dz)}^{1-\sigma}
\|\nabla v\|_{L^2(\R^n,dz)}^\sigma,
$$
which is, by using spherical coordinates, exactly the demanded inequality.
\end{proof}
As an immediate consequence of the latter theorem we obtain:
\begin{coro} Let $n\geq 3$.
\begin{enumerate}
\item The semigroup $T_t$ is ultra-contractive for every $t>0$. Moreover,
\begin{align}
p_t(x,y)\leq ct^{-n/2},\ \forall\,t,x,y>0.
\label{UpperHeat}
\end{align}
\item Every function from $\calD$ decays at infinity at most as $t^{-\frac{n-1}{2}}$.
\end{enumerate}
\end{coro}
\begin{proof}
According to \cite [Theorem 2.17]{Carlen}, Sobolev inequality leads to Nash inequality which in turn, according to    Theorem 2.1 \cite [Theorem 2.1]{Carlen} leads to the upper bound (\ref{UpperHeat}).\\
The proof of assertion 2) is a direct consequence from \ref{Ineq}-2).
\end{proof}

Let us turn our attention to prove an other global property for $\calE$, namely, conservativeness. We stress that the following result is known. We shall prove it utilizing the formula of the heat kernel.

\begin{theo}
The Dirichlet form $\calE$ is  conservative.
\label{Conservative}
\end{theo}
\begin{proof}
We shall prove $T_t 1=1$ for every $t>0$, where $T_t$ is the $L^\infty$ related to $\calE$. From the standard $L^\infty$-semigroup for $\calE$, the latter identity is equivalent to
$$
\int_0^\infty p_t(x,y)\,dm(y)=1.
$$
By monotone convergence theorem we get
\begin{eqnarray*}
\int_0^\infty p_t(x,y)\,dm(y)
&= & \int_0^\infty  \frac{1}{2 t}(xy)^{-\nu} \exp \big( -\frac{x^2 + y^2}{2 t}\big) I_{\nu}\big(\frac{xy}{t}) \,dm(y)\\
& = &  \int_0^\infty  \frac{1}{t}(xy)^{-\nu} \exp \big( -\frac{x^2 + y^2}{2 t}\big) I_{\nu}\big(\frac{xy}{t}) \, y^{2\nu+1} dy\\
& = &  \int_0^\infty  \frac{1}{t}(xy)^{-\nu} \exp \big( -\frac{x^2 + y^2}{2 t}\big)\sum_{k=0}^{\infty}\frac{(\frac{xy}{2t})^{2k+\nu}}{\Gamma(k+\nu+1)k!} \, y^{2\nu+1}\,dy\\
& = & e^{-\frac{x^2}{2t}}\sum_{k=0}^{\infty}\frac{ x^{2k}}{2^{2k+\nu}t^{2k+\nu+1}\Gamma(k+\nu+1)k!}\int_0^\infty  e^{-\frac{y^2}{2t}} y^{2(k+\nu)+1}\,dy.
\end{eqnarray*}
Let us recall the  Gamma function which is  defined by

$$\Gamma (z) :=   \int_0^\infty e^{-\tau}\ \tau^{z-1}\, d\tau,\ \forall z>0.  $$
Therefore, by using the change of variable $\tau = \frac{y^2}{2 t}$ we obtain
\begin{eqnarray*}
\int_0^\infty  e^{-\frac{y^2}{2t}} y^{2(k+\nu)+1}\, dy
& = & \int_0^\infty  e^{-\tau}\ (\sqrt{2 t \tau})^{2(k+\nu)}\ t d\tau\\
& = & 2^{k+\nu} t^{k+\nu+1} \Gamma(k+\nu+1).
\end{eqnarray*}
Finally, we achieve

\begin{eqnarray*}
\int_0^\infty p_t(x,y)\,dm(y)
&= & e^{-\frac{x^2}{2t}}\sum_{k=0}^{\infty}\frac{ x^{2 k}2^{k+\nu} t^{k+\nu+1} \Gamma(k+\nu+1)}{2^{2k+\nu}t^{2k+\nu+1}\Gamma(k+\nu+1)k!} \\
& = & e^{-\frac{x^2}{2t}}\sum_{k=0}^{\infty}\frac{ x^{2 k}}{(2 t)^k k!} = 1,
\end{eqnarray*}
yielding the conservativeness of $\calE$.
\end{proof}
It is well known that $\calE$ is transient. However, for the convenience of the reader we shall restate this property and proved with a different manner.
\begin{prop}
The form $\calE$ is transient.
\end{prop}
\begin{proof}
Let $g(t):=e^{-t},\ t\geq 0$. Then $g>0,\ g\in L^p(I,m)$ for each $1\leq p\leq\infty$. By Sobolev inequality in conjunction with H\"older inequality we get
$$
\int_0^\infty|u|g\,dm\leq c\sqrt{\calE[u]},\ \forall\,u\in\calD,
$$
yielding the transience of $\calE$.
\end{proof}

\section{Global properties of traces of Bessel's Dirichlet form  on discrete sets}
In this section we shall first compute the trace of $\calE$ w.r.t to discrete measures supported by $\N$ for the special case $n=3$. Then we proceed to  investigate conservativeness and irreducibility properties of of the obtained trace form. Thus from now on we fix
$$
n=3 \quad\text{ and hence } \nu= 1/2.
$$
We set $\N_0=\N\cup \{0\}$.
\begin{lem}
Let $(a_k)_{k\in\N_0}$ be a sequence of real numbers such that $a_k\geq 0$ for all $k\in\N_0$. Let $\mu$ be the measure defined on $\N_0$ by
\begin{eqnarray}
\mu=\sum_{k\in\N_0} a_k\delta_k.
\label{atomic}
\end{eqnarray}
Then $\mu$ is a smooth measure with respect to $\calE$ (i.e. $\mu$ does not charge any point having  zero $\calE$-capacity) with support $\N_0$ if and only if $a_0=0$ and $a_k>0$ for each $k\in\N$.
\end{lem}
\begin{proof}
It is known (see \cite[p. 339]{Yor}) that $Cap(\{ 0 \})=0$. On the other hand by inequality (\ref{PointK}) we have
$$
k|u(k)|\leq c\sqrt{\calE[u]},\ \forall\,u\in\calD.
$$
Hence $Cap(\{ k \})\geq k^2/{c^2}$ for every $k\in\N$. Accordingly, $\mu$ is smooth if and only if $a_0=0$. Now the condition that the quasi-support of $\mu$ coincides with $\N_0$ is equivalent to $a_k>0$ for each $k\in\N$.
\end{proof}
In this section we fix a discrete measure
\[ 
\mu=\sum_{k\in\N} a_k\delta_k\ \text{ with }\ a_k>0,\ \forall\, k\in\N.
\]
Let $\check\calE$ be the trace of $\calE$ w.r.t $\mu$ (or on the set $\N$) (See \cite{Fukushima,Chen-Fukushima,BBST})). In order to compute $\ccalE$ we shall adopt the method developed in \cite{BBST}. Let us be more concrete and describe the strategy we shall follow toward computing $\ccalE$. Let $J$ be the restriction operator from $\calD$ to  $ L^2 (\N, \mu)=\ell^2(\mu)$ defined  as follows
\[
    D(J):=\{u\in \calD:\; \sum_{k\in\N} a_k u(k)^2 < \infty\},\quad Ju:=u|_{\N} \quad\text{for all }u\in D(J),
\]
It is easy to check that $J$ is closed in $(\calD,\calE_1)$. Moreover the regularity property for $\calE$ implies that $J$ has dense range. Obviously the kernel of $J$ is
$$
\ker J = \{ u\in \calD : u(k) = 0, \forall k\in \N \}.
$$
For every $u\in\calD,\ \lambda>0$, let $P_\lambda u$ be the orthogonal projection from the Dirichlet space $(\calD, \calE_\lambda)$ onto the orthogonal complement of $\ker J $ w.r.t. the scalar product $\calE_\lambda$. For each $\lambda>0$, we define a quadratic form $\ccalE_\lambda$ as follows:
\[
\ccalE_\lambda[Ju]:=\calE_\lambda[P_\lambda u],\ \forall\,u\in D(J).
\]
From \cite[Theorem 1]{BBBMath}, the form $\ccalE_\lambda$ is closed in $\ell^2(\mu)$ for each $\lambda>0$. Moreover by \cite[Theorem 2.1]{BBST} the family $(\ccalE_\lambda)_{\lambda>0}$  is monotone increasing. By definition, $\ccalE$ is the Mosco limit of $(\ccalE_\lambda)$ as $\lambda\downarrow 0$ (see \cite[Theorem 2.4]{BBST}).
\subsection{Computing $\ccalE$}
At first stage we shall establish an explicit formula for the approximating forms $\ccalE_\lambda$ for each $\lambda>0$.
\begin{lem}
Let $u\in D(J)$. Then $P_\lambda u$ is the unique function from $\calD$ which solves the differential equation
\begin{align}
\label{ModBessel}
    -\frac{1}{2}(P_\lambda u)'' -\frac{1}{x}(P_\lambda u)' +\lambda P_\lambda u & = 0 \quad\text{ in } \ (0,\infty)\setminus\N,\nonumber\\
    P_\lambda u & = u \quad\text{ on } \N.
\end{align}
\label{Projection}
\end{lem}
\begin{proof} Let $u\in D(J)$ and $\lambda>0$.  Regarding the definition of $P_\lambda u$ we obtain
 $\calE_\lambda( P_\lambda u, v) =0 $ for every $v$ in $\mathcal{C}_c^{\infty}((0,\infty)\setminus \N)$, which is equivalent to
\begin{equation}
\int_0^\infty (P_\lambda u)'(x)v'(x) x^{2}\,dx + \lambda\int_0^\infty (P_\lambda u)(x) v(x) 2 x^2 \, dx = 0. \  \forall v \in \mathcal{C}_c^{\infty}((0,\infty)\setminus \N).
\end{equation}
Hence
\begin{align}
-\frac{1}{2}(P_\lambda u)'' -\frac{1}{x}(P_\lambda u)' +\lambda P_\lambda u & = 0 \quad\text{ in the sense of distributions in } \ (0,\infty)\setminus\N.
\label{Distribution}
\end{align}
Being solution of an ODE, with smooth coefficients on  $(0,\infty)\setminus\N$ we conclude that $P_\lambda u\in \mathcal{C}^{\infty}((0,\infty)\setminus \N)$ and hence equation (\ref{Distribution}) is fulfilled pointwise on $(0,\infty)\setminus\N$.\\
As for the boundary condition we have $ u-P_\lambda u \in Ker(J)^{\perp\perp}$. Since $J$ is a closed operator then its kernel $Ker(J)$ is also closed and hence $ u-P_\lambda u \in Ker(J)$. This implies that $Ju = JP_\lambda u$ and hence $u= P_\lambda u \ \mu-a.e. $ on $\N$.
The proof of the converse is easy so we omit it.
\end{proof}
The differential equation given in (\ref{ModBessel}) is in fact equivalent to
$$
(P_\lambda u)'' +\frac{2}{x}(P_\lambda u)' -2\lambda P_\lambda u  = 0,
$$
which is nothing else but a modified Bessel differential equation. Hence, the general solution of the latter equation is given by ( see \cite[p. 362, Eq. 9.1.52]{Abramowitz} )
\begin{eqnarray}
\text{for all} \ k\in \N,  P_\lambda u(x) = x^{- 1/2 }A_{k}\ I_{1/2}(x\sqrt{2\lambda})+  x^{-1/2}B_{k}\ K_{1/2}(x\sqrt{2\lambda}),\ \mbox{in} \ [k,k+1],
   \label{GenSolution}
\end{eqnarray}
and
\begin{eqnarray}
P_\lambda u(x) = x^{- 1/2 }A_{0}\ I_{1/2}(x\sqrt{2\lambda})+  x^{-1/2}B_{0}\ K_{1/2}(x\sqrt{2\lambda}),\ \mbox{in} \ (0,1].
\end{eqnarray}

Here $A_{k},B_{k}$ are real constants to be adjusted according to the boundary conditions and $I_{1/2}, K_{1/2}$ are the  modified Bessel functions given by
$$
I_{\frac{1}{2}}(x\sqrt{2\lambda})= \sqrt{\frac{2}{\pi x\sqrt{2\lambda}}}\sinh (x\sqrt{2\lambda})\ \text{and}\ \  K_{1/2}(x\sqrt{2\lambda})= \sqrt{\frac{\pi}{2 x\sqrt{2\lambda}}} e^{-x\sqrt{2\lambda}},
$$
For later computations we set $M_k$ the matrix
\begin{align}
 M_k := M_k(\lambda) := \left(\begin{array}{cc}
\frac{I_{1/2}(k\sqrt{2\lambda})}{k^{ 1/2 }}  &   \frac{ K_{1/2}(k\sqrt{2\lambda})}{k^{ 1/2 }}  \\
\\
\frac{I_{1/2}((k+1)\sqrt{2\lambda})}{(k+1)^{ 1/2 }} &  \frac{ K_{1/2}((k+1)\sqrt{2\lambda})}{(k+1)^{ 1/2 }}
\end{array}\right),\ \forall\,k\in\N.
\end{align}
An elementary computation leads to evaluate the determinant of $M_k$:
\begin{align}
det(M_k) = - \frac{\sinh(\sqrt{2\lambda})}{\sqrt{2\lambda}\ k(k+1) }.
\label{Deter}
\end{align}
Hence $M_k$ is invertible for each $\lambda>0$.
\begin{lem}
\begin{enumerate}
\item It holds $B_0=0$, $A_0 = \frac{u(1)}{I_{1/2}(\sqrt{2\lambda})}$.
\item For each $k\in\N$ it holds
$$
A_k =  \frac{1}{det(M_k)} \left( u(k)\  \frac{ K_{1/2}((k+1)\sqrt{2\lambda})}{(k+1)^{ 1/2 }} - u(k+1)\  \frac{ K_{1/2}(k\sqrt{2\lambda})}{k^{ 1/2 }}\right),
$$
and
$$
B_k =  \frac{1}{det(M_k)} \left( u(k+1)\   \frac{I_{1/2}(k\sqrt{2\lambda})}{k^{ 1/2 }} - u(k)    \frac{I_{1/2}((k+1)\sqrt{2\lambda})}{(k+1)^{ 1/2 }}  \right).
$$
\end{enumerate}
\label{Coefficients}
\end{lem}
\begin{proof}
{\em The case} $k=0$. Since the function $x^{-1/2}K_{1/2}$ is singular at $0$ whereas $P_\lambda u$ should be bounded near $0$ we obtain $B_0=0$. The computation of $A_0$ is easy. Indeed, we have
$$  P_\lambda u(x) = A_0 \frac{I_{1/2}(x \sqrt{2\lambda})}{x^{1/2}},\  \text{in\  }   [0,1] .    $$
Taking into account the boundary condition at $x=1$, we get the desired result.\\
{\em The case}  $k\in\N$. To determine the coefficients $A_k$ and $B_k$ we have to adapt the general solution given by (\ref{GenSolution}) to the boundary conditions of Lemma \ref{Projection}. Namely,
taking the boundary conditions in (\ref{ModBessel}) into account we derive
\begin{align*}
P_\lambda u(k)&=  A_{k}\ \frac{I_{1/2}(k\sqrt{2\lambda})}{k^{1/2 }}+ B_{k}\ \frac{K_{1/2}(k\sqrt{2\lambda})}{k^{1/2 }}=u(k)\\
P_\lambda u(k+1) &=  A_{k}\ \frac{I_{1/2}((k +1) \sqrt{2\lambda})}{(k +1) ^{1/2 }}+  B_{k}\ \frac{K_{1/2}((k +1)\sqrt{2\lambda})}{(k +1) ^{1/2 }}=u(k+1).
\end{align*}
The latter linear system is equivalent to

$$  M_k\left(\begin{array}{c} A_k\\
\\
B_k
\end{array}\right)=   \left(\begin{array}{c} u(k)\\
\\
u(k+1)
\end{array}\right),$$
which leads to the formula to be proved.

\end{proof}
\begin{lem}
For every $u\in\calD$ and every $\lambda>0$, it holds
\begin{eqnarray}
\check{\calE_\lambda}[Ju]&=&\sum_{k=0}^{\infty} \left[ (P_\lambda u)'((k+1)^-) u(k+1)(k+1)^2 - (P_\lambda u)'(k^+) u(k) k^2\right] \nonumber\\
& = &
\sum_{k=1}^{\infty} \Big( -\frac{ u(k)u(k+1)  \sqrt{2\lambda} \ k(k+1) }{  \sinh(\sqrt{2\lambda})} + \frac{ u(k+1)^2 \ (k+1)^2 \sqrt{2\lambda} \cosh(\sqrt{2\lambda}) }{ \sinh(\sqrt{2\lambda}) }\nonumber\\
&& - u(k+1)^2 (k+1)
- u(k+1)u(k) k(k+1)  \frac{\sqrt{2\lambda}}{\sinh(\sqrt{2\lambda})}\nonumber\\
&&+(\frac{ k^2  \sqrt{2\lambda}}{  \sinh(\sqrt{2\lambda})} \cosh(\sqrt{2\lambda})  + k )u(k)^2\Big) + \frac{\sqrt{2\lambda} \ u(1)^2  \ I_{3/2}(\sqrt{2\lambda})}{I_{1/2}(\sqrt{2\lambda})}.
\end{eqnarray}
\label{ApproxForm}

\end{lem}
\begin{proof}
Let $u\in\calD,\ \lambda>0$. Making use of Lemma \ref{Projection}, a straightforward computation leads to
\begin{align}
&\check{\calE_\lambda}[Ju]\nonumber =  \calE_\lambda[P_\lambda u]\\\nonumber
& =  \int_{0}^{\infty}((P_\lambda u)'(x))^{2}x^{2} dx+ \lambda \int_{0}^{\infty}(P_\lambda u)^{2}(x)\ 2 x^{2} dx\nonumber\\
& = \sum_{k=0}^\infty \left( \int_{k}^{k+1}((P_\lambda u)')^{2}(x) x^{2} dx+ \lambda \int_{k}^{k+1}(P_\lambda u)^{2}(x) 2 x^{2} dx\right)\nonumber\\
& = \sum_{k =0}^\infty ( \int_{k}^{k+1}\left(-\frac{1}{2}(P_\lambda u)''(x)(P_\lambda u)(x) -\frac{1}{x}(P_\lambda u)'(x)(P_\lambda u)(x)+\lambda (P_\lambda u)^{2}(x)\right)2x^{2} dx\nonumber\\
& +\sum_{k =0}^\infty (P_\lambda u)'(x) (P_\lambda u)(x)\ x^{2}|_{k}^{k+1}\nonumber\\
& =  \sum_{k\in \N} [(P_\lambda u)'((k+1)^- )u(k+1)(k+1)^2-(P_\lambda u)'(k^+) u(k)k^2]+ \frac{\sqrt{2\lambda} \ u(1)^2  \ I_{3/2}(\sqrt{2\lambda})}{I_{1/2}(\sqrt{2\lambda})},
\label{computeEL}
\end{align}
and the first identity of the lemma is proved.\\
Let us prove the second identity of the lemma. Clearly we are led to know $(P_\lambda u)'(k^+)$ and $(P_\lambda u)'((k+1)^-)$ (the right derivative at $k$ and the left derivative at $k+1$). In order to compute $(P_\lambda u)'$ let us recall the well known derivation formulae ( see \cite[p. 376, Eq. 9.6.28]{Abramowitz} )
\begin{eqnarray*}
\frac{\partial }{\partial x}( x^{-\nu}I_{\nu}(m x)) = m   x^{-\nu}I_{\nu+1}(m x),\
   \frac{\partial }{\partial x}( x^{-\nu}K_{\nu}(m x)) =  -m   x^{-\nu}K_{\nu+1}(m x).
\end{eqnarray*}
Having the latter formulae in hands together with the expression of  $P_\lambda u$ from (\ref{GenSolution}) we get
\begin{eqnarray*}
(P_\lambda u)'(x)
& = & A_{k}\frac{\partial }{\partial x}( x^{-1/2}\ I_{1/2}(x\sqrt{2\lambda}))+ B_{k} \frac{\partial }{\partial x}(x^{-1/2}\ K_{1/2}(x\sqrt{2\lambda})) \\
\\
& = & \sqrt{2\lambda}\ \left(   A_{k}\ x^{-1/2}\ I_{3/2}(x\sqrt{2\lambda}) - B_{k}\  x^{-1/2}\ K_{3/2}(x\sqrt{2\lambda})\right),\ \quad\text{ in }\ [k,k+1].
\end{eqnarray*}
We also recall the well known formulae
$$
I_{3/2}(x\sqrt{2\lambda}) = \sqrt{\frac{2}{\pi\ x\sqrt{2\lambda}}}  \left(   \cosh(x\sqrt{2\lambda})- \frac{ \sinh(x\sqrt{2\lambda})}{(x\sqrt{2\lambda})} \right)
$$
and
$$  K_{3/2}(x\sqrt{2\lambda}) =  \sqrt{\frac{2}{\pi\ x\sqrt{2\lambda}}}  \left( 1 + \frac{1}{x\sqrt{2\lambda}} \right) e^{-x\sqrt{2\lambda}}.
$$
Let us define the Wronskian of two modified  Bessel functions as follows
\begin{equation}
 W[K_{\nu}(x) ,  I_{\nu}(x)] :=  I_{\nu}(x) K_{\nu +1 } (x) +    I_{\nu +1}(x) K_{\nu}(x) =  \frac{1}{x}, \ x>0.
\end{equation}
Then, for all  $x , \lambda >0 $ we have
$$
 W[K_{1/2}(x\sqrt{2\lambda}) ,  I_{1/2}(x\sqrt{2\lambda})] = \frac{1}{x\sqrt{2\lambda}}.
$$
A lengthy computation leads to
\begin{eqnarray*}
(P_\lambda u)'(k^+)
& = & \frac{\sqrt{2\lambda}}{k^{1/2}}\ \left(   A_{k}\ \ I_{3/2}(k\sqrt{2\lambda}) - B_{k}\ K_{3/2}(k\sqrt{2\lambda})\right)\\
& = & \frac{\sqrt{2\lambda}\ I_{3/2}(k\sqrt{2\lambda})}{k^{1/2}\det(M_k)} \left(  u(k)\  \frac{ K_{1/2}((k+1)\sqrt{2\lambda})}{(k+1)^{ 1/2 }} - u(k+1)\  \frac{ K_{1/2}(k\sqrt{2\lambda})}{k^{ 1/2 }}\right) \\
&& -\frac{\sqrt{2\lambda}\ K_{3/2}(k\sqrt{2\lambda})}{k^{1/2}\det(M_k)} \left(  u(k+1)\   \frac{I_{1/2}(k\sqrt{2\lambda})}{k^{ 1/2 }} - u(k)    \frac{I_{1/2}((k+1)\sqrt{2\lambda})}{(k+1)^{ 1/2 }}  \right)\\
& = & \frac{\sqrt{2\lambda}\ u(k)}{k^{1/2}\det(M_k) (k+1)^{ 1/2 }}. \frac{1}{ \sqrt{2\lambda}\ k^{1/2}(k+1)^{1/2}}
\left( \cosh(\sqrt{2\lambda}) + \frac{\sinh(\sqrt{2\lambda})}{k \ \sqrt{2\lambda}} \right) \\
& &  -\frac{\sqrt{2\lambda}\ u(k+1)}{k^{1/2}\det(M_k) k^{ 1/2 }}.  \frac{1}{  k \ \sqrt{2\lambda}}\\
& = & -\frac{  \sqrt{2\lambda}\ u(k)}{  \sinh(\sqrt{2\lambda})}\left( \cosh(\sqrt{2\lambda})  + \frac{ \sinh(\sqrt{2\lambda})}{ k \ \sqrt{2\lambda}} \right) + \frac{  u(k+1) \sqrt{2\lambda}\ (k+1)}{ k \sinh(\sqrt{2\lambda})}\\
& = &  \frac{ u(k+1) (k+1)}{ k } \frac{\sqrt{2\lambda}}{\sinh(\sqrt{2\lambda})}-\left(\frac{  \sqrt{2\lambda}}{  \sinh(\sqrt{2\lambda})} \cosh(\sqrt{2\lambda})  + \frac{1}{ k} \right)u(k).
\end{eqnarray*}
Finally we obtain,
$$
(P_\lambda u)'(k^+) u(k) k^2  =
 u(k+1)u(k) k(k+1)  \frac{ \sqrt{2\lambda}}{ \sinh(\sqrt{2\lambda})}-k\ u(k)^2 \left(\frac{  k  \sqrt{2\lambda}\cosh(\sqrt{2\lambda})}{  \sinh(\sqrt{2\lambda})}   +1 \right).
$$
Similarly we get
\begin{eqnarray*}
(P_\lambda u)'((k+1)^-)
 = -\frac{ u(k)  \sqrt{2\lambda} \ k }{ (k+1)\  \sinh(\sqrt{2\lambda})} +  \frac{  u(k+1)\ \sqrt{2\lambda} \cosh(\sqrt{2\lambda}) }{   \sinh(\sqrt{2\lambda}) } - \frac{  u(k+1)}{ (k+1)\  }.
\end{eqnarray*}
and
\begin{align*}
 (P_\lambda u)'((k+1)^-)u(k+1)(k+1)^2
 & = k(k+1) u(k)u(k+1) \frac{  \sqrt{2\lambda} }{   \sinh(\sqrt{2\lambda})} \\
 & +  u(k+1)^2 (k+1) \left( \frac{(k+1)\ \sqrt{2\lambda} \cosh(\sqrt{2\lambda}))}{\sinh(\sqrt{2\lambda})}-1\right).
\end{align*}
Substituting in  (\ref{computeEL}) we get
\begin{eqnarray}
\check{\calE_\lambda}[Ju]&=&\sum_{k\geq 1} (P_\lambda u)'((k+1)^-) u(k+1)(k+1)^2 - (P_\lambda u)'(k^+) u(k) k^2\nonumber\\
  &+& \frac{\sqrt{2\lambda} \ u(1)^2  \ I_{3/2}(\sqrt{2\lambda})}{I_{1/2}(\sqrt{2\lambda})}\nonumber\\
& = &
\sum_{k\geq 1}\Big( -\frac{ u(k)u(k+1)  \sqrt{2\lambda} \ k(k+1) }{  \sinh(\sqrt{2\lambda})} + \frac{ u(k+1)^2 \ (k+1)^2 \sqrt{2\lambda} \cosh(\sqrt{2\lambda}) }{ \sinh(\sqrt{2\lambda}) }\nonumber\\
&& - u(k+1)^2 (k+1)
- u(k+1)u(k) k(k+1)  \frac{\sqrt{2\lambda}}{\sinh(\sqrt{2\lambda})}\nonumber\\
&&+(\frac{ k^2  \sqrt{2\lambda}}{  \sinh(\sqrt{2\lambda})} \cosh(\sqrt{2\lambda})  + k )u(k)^2\Big) + \frac{\sqrt{2\lambda} \ u(1)^2  \ I_{3/2}(\sqrt{2\lambda})}{I_{1/2}(\sqrt{2\lambda})},
\end{eqnarray}
and the proof is finished.
\end{proof}

We are in position now to compute the trace from $\ccalE$ through an approximation procedure as explained above.
\begin{theo}
For every $u=(u_k)\in\ \ell^2(\N,\mu)$, set
$$
Q[u]=\sum_{k=1}^\infty k(k+1)( u_{k+1} - u_k )^2.
$$
\begin{enumerate}
\item The trace form $\ccalE$ is the closure of $Q$ restricted to $\ran J$.
\item Assume that  $\mu(\N)=\infty$. Then
\begin{align*}
\check{\calD}:= D(\ccalE)=\{ u=(u_k)\in\ell^2(\N,\mu),\  Q[u]<\infty\},\ \ccalE[u]=Q(u),\ \forall\,u\in \check{\calD}.
\end{align*}
\end{enumerate}
\label{TraceForm}
\end{theo}
\begin{proof}
1. Let $u\in\calD$. Letting $\lambda\downarrow 0$, we obtain by using the monotone convergence theorem for series
\begin{eqnarray}
\ccalE_0[Ju]:=\lim_{\lambda\downarrow 0}\check{\calE_\lambda}[Ju]&=&\sum_{k=1}^\infty k(k+1)  \left( u(k+1) -  u(k) \right)^2.
\end{eqnarray}
It is easy to check that the limit form $\ccalE_0$ with domain $\ran J$ is closable in $\ell^2(\N,\mu)$. Hence by \cite[Theorem 2.4]{BBST}, $\ccalE$ is the closure of  $\ccalE_0$. Observing that $\ccalE_0=Q|_{\ran J}$ leads to the claim.\\
Let  us write
\[
k\sim j\Leftrightarrow |k-j|=1,
\]
and
\[
b(k,j)= \frac{1}{2}k j,\ \text{ if }\ k\sim j\ \text{ and } b(k,j)=0\ \text{ otherwise}.
\]
Then
\begin{align}
 \ccalE_0[Ju]=\sum_{k\in\N}\sum_{k\sim j}b(k,j)(u(k) - u(j))^2,\ \forall\,u\in\calD.
 \label{Beurling}
\end{align}
2. Now assume $\mu(\N)=\infty$. From formula (\ref{Beurling}), we infer $b(k,k+1)>0$ for all $k$ and $b(k,j)=0$ for $|k-j|>1$. Thereby, condition (A) from \cite{Keller-Lenz} is fulfilled. Moreover $\tilde L (C_c(\N))\subset C_c(\N)$ where
\[
\tilde L u(k) := \frac{1}{a_k} \sum_{j} b(k,j) (u(k)-u(j)), \ \text{for each}\ k\in\N.
\]
Hence assertion 2. is a corollary of \cite[Theorem 6]{Keller-Lenz}.
\end{proof}
\subsection{Global properties of $\ccalE$}

\begin{theo}
The trace form $\ccalE$ is  irreducible and transient.
\end{theo}
\begin{proof}
{\em Irreducibility}. Assume there is a set $\emptyset\neq Y\subset\N$ such $Y\neq\N$ and $Y$ is invariant. Then there is $N\in\N$ such that $N\in Y$ and $N+1\not\in Y$. Set $u=\mathbf{1}_{\{N,N+1\}}$. Then $u, u\mathbf{1}_Y$  and $u\mathbf{1}_{Y^c}$ are in $\ccalD$ (in fact they are in $\ran J$ because they all have finite supports). The irreducibility of $Y$ should yield $\ccalE[u]=\ccalE[u\mathbf{1}_Y ] +\ccalE[ u\mathbf{1}_{Y^c}]$. However, $\ccalE[u\mathbf{1}_Y ]=2N^2,\ \ccalE[ u\mathbf{1}_{Y^c}]=2N^2+3N$ and $\ccalE[u]=2N^2+2N\neq \ccalE[u\mathbf{1}_Y ] +\ccalE[ u\mathbf{1}_{Y^c}]$. Thus $\ccalE$ is irreducible.\\
{\em Transience}. Let $u\in\calD$. Using inequality (\ref{Strauss}) we obtain
\begin{align}
(P_\lambda u (k))^2= (u(k))^2\leq \frac{c}{k}\calE_\lambda[P_\lambda u],\ \forall\,k\in\N.
\label{PointK}
\end{align}
Let $g:\N\to (0,\infty),\ g(k)=e^{-k}/{a_k}$. Then $g>0$ and $g\in \ell^1(\N,\mu)$. Using inequality (\ref{PointK}) we get
\begin{align}
\int_{\N} |P_\lambda u(k)|g(k)\,d\mu&=\int_{\N} |u(k)|g(k)\,d\mu= \sum_{k\geq 1} |u(k)|g(k)a_k\nonumber\\
&\leq c\sqrt{\calE_\lambda[P_\lambda u]}.
\end{align}
The latter inequality reads
\begin{align}
\int_{\N} |Ju (k)|g(k)\,d\mu\leq c\sqrt{\ccalE_\lambda[Ju]},\ \forall\,u\in\calD.
\end{align}
Thus letting $\lambda\downarrow 0$ and taking into account that $\ran J$ is a form core for $\ccalE$ we achieve the inequality
\begin{align}
\int_{\N} |u (k)|g(k)\,d\mu\leq c\sqrt{\ccalE[u]},\ \forall\,u\in\ccalD.
\end{align}
Thereby $\ccalE$ is transient, and the proof is finished.
%
\end{proof}
Let us now discuss conservativeness of $\ccalE$.
\begin{theo}
\label{NSC-BesselCons}
\begin{enumerate}
\item Assume that $\mu$ is finite. Then the  trace form of the Bessel process is not conservative.
\item Assume that $\mu$ is infinite. Then the  trace form of the Bessel process is conservative if and only if
\begin{eqnarray}
\sum_{k=1}^\infty \frac{a_k}{k} = \infty.
\end{eqnarray}
\label{BesselCons}
\end{enumerate}
\end{theo}
\begin{proof}

{\em Case 1:} $\mu$ is finite. In this situation conservativeness and recurrence are equivalent. However we have already proved that $\ccalE$ is transient and then it is not recurrent and hence not conservative.\\
{\em Case 2:} $\mu$ is infinite. In this situation we use \cite[Theorem 6.1]{Keller-Lenz} ($Q=Q^{max}$ if the measure is infinite) to conclude that the conservativeness of $\ccalE$ is equivalent to the fact that the equation
\begin{eqnarray}
\tilde{L}u + \alpha u=0,\ \alpha>0,\ u\in l^\infty,
\end{eqnarray}
has no nontrivial bounded  solution. Here $u=(u_k)$. We rewrite
\begin{eqnarray}
\tilde{L}u(k) + \alpha u(k)=\frac{1}{a_k}\sum_{j} b(k,j)(u_k - u_j) + \alpha u_k=0.
\end{eqnarray}
This leads to,
\begin{eqnarray}
 u_2= (1 + \alpha a_1)u_1,
\end{eqnarray}
and
\begin{eqnarray}
\frac{1}{2a_k}k(k+1)(u_k - u_{k+1}) + \frac{1}{2a_k}k(k-1)(u_k - u_{k-1}) + \alpha u_k=0,\ \forall\,k\geq 2.
\end{eqnarray}
Thus by induction we get
\begin{eqnarray*}
u_{k+1} - u_k &=&   \frac{k-1}{k+1}(u_k - u_{k-1}) + 2\frac{a_k\alpha}{k(k+1)} u_k\nonumber\\
&=& \frac{k-1}{k+1}\big(\frac{k-2}{k}(u_{k-1} - u_{k-2}) + 2\frac{a_{k-1}\alpha}{k(k-1)} u_{k-1}\big)
+ 2\frac{a_k\alpha}{k(k+1)} u_k\nonumber\\
&=& \frac{(k-1)(k-2)}{k(k+1)}(u_{k-1} - u_{k-2}) +  2\frac{a_{k-1}\alpha}{k(k+1)} u_{k-1}
+ 2\frac{a_k\alpha}{k(k+1)} u_k\nonumber\\
&=& \frac{(k-1)(k-2)}{k(k+1)}\big( \frac{k-3}{k-1} (u_{k-2} - u_{k-3}) +\cdots\big) +  2\frac{a_{k-1}\alpha}{k(k+1)} u_{k-1}\nonumber\\
&+& 2\frac{a_k\alpha}{k(k+1)} u_k,\ \forall\,k\geq 2.
\end{eqnarray*}
Finally we obtain the recursive formula
\begin{eqnarray}
u_{k+1} - u_k &=& \frac{u_2 - u_1}{k(k+1)} + \frac{2\alpha}{k(k+1)}\sum_{j=1}^k a_j u_j\nonumber\\
&=& \frac{\alpha}{k(k+1)}a_1u_1 + \frac{2\alpha}{k(k+1)}\sum_{j=1}^k a_j u_j
,\ \forall\,k\geq 2.
\label{difference}
\end{eqnarray}
The latter formula leads to the following two observations (which can be proved by induction):\\
1. $u_k$ has the sign of $u_1$ for all $k\geq 1$. This is if $u_1>0$, then $u_k>0,\ \forall\ k\geq 1$ and if  $u_1<0$, then $u_k<0,\ \forall\ k\geq 1$.\\
2. The sequence $(u_k)$ is monotone, depending on the sign of $u_1$.\\
Thus by linearity we may assume, without loss of generality, that $u_1>0$. In this case $(u_k)$ is strictly monotone increasing.\\
Accordingly, making use of formula (\ref{difference}) we derive
\begin{eqnarray}
u_{k+1} - u_k \leq\big( \frac{\alpha a_1}{k(k+1)} + \frac{2\alpha}{k(k+1)}\sum_{j=1}^k a_j\big)u_k
,\ \forall\,k\geq 2,
\label{domination1}
\end{eqnarray}
and
\begin{eqnarray}
\frac{ u_{k+1}}{u_k} \leq 1 + \frac{\alpha a_1}{k(k+1)} + \frac{2\alpha}{k(k+1)}\sum_{j=1}^k a_j
,\ \forall\,k\geq 2.
\end{eqnarray}
Finally we achieve
\begin{eqnarray}
u_{N+1}\leq u_2\prod_{k=1}^{N+1} \big(1 + \frac{\alpha a_1}{k(k+1)} + \frac{2\alpha}{k(k+1)}\sum_{j=1}^k a_j\big)
\end{eqnarray}
Obviously the latter product is finite provided $\sum_{k=1}^\infty \frac{1}{k(k+1)}\sum_{j=1}^k a_j<\infty$ and then we get a bounded non-zero solution.\\
Conversely Assume that $\sum_{k=1}^\infty \frac{1}{k(k+1)}\sum_{j=1}^k a_j=\infty$. Then summing over $k$ in formula (\ref{difference}) and having in mind that $(u_k)$ is increasing we obtain
\begin{eqnarray}
u_{N+1}- u_1 = a_1u_1\alpha\sum_{k=1}^N \frac{1} {k(k+1)} +
2\alpha\sum_{k=1}^N \frac{1}{k(k+1)}\sum_{j=1}^k a_j u_j.
\label{Sum}
\end{eqnarray}
Hence
$$
u_{N+1}\geq 2\alpha a_1\sum_{k=1}^N \frac{1}{k(k+1)}\sum_{j=1}^k a_j\to\infty \mbox{ as}\ N\to\infty.
$$
Finally an elementary computation show that $\sum_{k=1}^\infty \frac{1}{k(k+1)}\sum_{j=1}^k a_j = \sum_{k\geq 1} \frac{a_k}{k}$, which finishes the proof.
\end{proof}
\begin{rk}
{\rm
Let us emphasize that for this special case condition (14) of \cite{GrigoryanStochGraph} is not fulfilled. It also does not fit in the framework of  \cite[Section 5]{MasamuneConserv}.
}
\end{rk}
We close this section by analyzing global properties of elements from $\ccalD$ as well as compactness of the embedding of $\ccalD$ into $\ell^p(\N,\mu)$.
Obviously $\ccalD\subset \ell^\infty$. We shall perform this observation by establishing decay property for elements from $\ccalD$ as $k\to\infty$. Let us quote that such result enables us to describe decay of eigenfunctions of $\check{H}$.
\begin{prop}
There is  a finite constant $c>0$ such that
$$
|u_k|\leq \frac{c}{k^{1/2}}(\ccalE[u])^{1/2},\ \forall\,u=(u_k)\in\ccalD.
$$
\end{prop}
\begin{proof}
Let $u\in\calD$. Using inequality (\ref{Strauss}), with $\sigma=1$, we get
$$
k^{1/2} \ |P_\lambda u(k)| =  k^{1/2} \ |u(k)| \leq c\  (\calE[ P_\lambda u])^{1/2}\leq c\  (\calE_\lambda [ P_\lambda u])^{1/2},\  \forall k\in\N,\ \lambda>0.
$$
Thus,
$$
|(Ju)(k)| \leq \frac{c}{k^{1/2}}\  (\ccalE_\lambda [J u])^{1/2},\  \forall k\in\N ,\ \forall\,u\in\calD.            $$
Letting $\lambda \downarrow 0$, and using the fact that $\ran J$ is a core for $\ccalE$, we obtain
$$
|u_k|\leq \frac{c}{k^{1/2}}(\ccalE[u])^{1/2},\  \forall k\in\N ,\ \forall\,u=(u_k)\in\ccalD,
$$
and the proof is finished.

\end{proof}
Let us stress that on the light of the latter proposition the decay behavior is preserved under taking the trace.
\begin{prop}
Let $p\in [1,\infty)$ be such that
$$
\sum_{k\in\N} \frac{a_k}{k^{p/2}}<\infty.
$$
Then $\ccalD$ embeds compactly into $l^p(\N,\mu)$.

\end{prop}
\begin{proof}
Set $q=p/2$. Let $u=(u_k)\in\ell^2(\N,\mu)$. Then by
\[
\sum_k |u_k|^p a_k \leq c \big(\sum_k \frac{a_k}{k^q}\big)\cdot(\ccalE[u])^q.
\]
Now let $(v_j)=(u_k^{(j)})\subset\ccalD$ which converges $\ccalE_1$-weakly to $0$. From the latter inequality we get
\begin{eqnarray}
\sup_j\|v_j\|_{\ell^p(\N,\mu)}<\infty.
\label{UniformBound}
\end{eqnarray}
Let $\epsilon$. Choose $N\in\N$  large enough so that
$$
\sum_{k\geq N} \frac{a_k}{k^q}<\epsilon.
$$
Owing to the $\ccalE_1$ boundedness of the sequence $(v_j)$ we get
\begin{align}
\|v_j\|_{\ell^p(\N,\mu)}^p &= \sum_{k=1}^{ N -1} |u_k^{(j)}|^p a_k + \sum_{k\geq N} |u_k^{(j)}|^p a_k\nonumber\\
&\leq \sum_{k=1}^{ N -1} |u_k^{(j)}|^p a_k + \big(\sum_{k\geq N }\frac{a_k}{k^q}\big)\cdot(\ccalE[v_j])^q
 \leq \sum_{k=1}^{ N -1} |u_k^{(j)}|^p a_k +C\epsilon.
 \label{Limit1}
\end{align}
From the uniform bound (\ref{UniformBound}), we derive that for every $k=1,\cdots,N-1$ the sequence $(u_k^{(j)})_j$ is uniformly bounded in $\R$ and hence each of then  has a convergence subsequence. Since they are finite in number we may and shall assume without loss of generality that they have a common convergent  subsequence say $(u_{k}^{(j_l)})_l$. As by assumptions $(v_j)$ converges $\ccalE_1$-weakly to zero we obtain
$$
\lim_{l\to\infty} u_{k}^{(j_l)} =0,\ \forall\,k=1,\cdots,N-1.
$$
Thus by (\ref{Limit1}) we get
\[
\limsup_{l\to\infty} \|v_j\|_{\ell^p(\N,\mu)}^p \leq C\epsilon.
\]
As $\epsilon$ is arbitrary, we obtain  $\|v_j\|_{\ell^p(\N,\mu)}\to 0$ and therefore $\ccalD$ embeds compactly into $\ell^p(\N,\mu)$.

\end{proof}
\subsection{The trace on finite sets}
We consider now an atomic measure with finite support:
\begin{align}
\mu=\sum_{k=1}^N a_k\delta_k,\ N\in\N.
\end{align}
Unlike the former case where $\ccalE$ is of pure jump type, in this case we shall show that the trace of $\calE$ on the set $\{1,2,\cdots,N\}$ decomposes into the sum of a nonlocal an a killing Dirichlet forms.\\
Let us first compute $\ccalE$. For this situation
$$
\ker J=\{u\in\calD\colon\,u(k)=0,\ k=1,\cdots,N\}
$$
Let $u\in\calD$ and $\lambda>0$. Then $P_\lambda u$ is the unique function from $\calD$ which solves the differential equation
\begin{align}
    -\frac{1}{2}(P_\lambda u)'' -\frac{1}{x}(P_\lambda u)' +\lambda P_\lambda u & = 0 \quad\text{ in } \ (0,\infty)\setminus\{1,...,N\},\nonumber\\
    P_\lambda u & = u \quad\text{ on } \{1,...,N\},
\end{align}
Hence owing to the decay property of $P_\lambda u$ at infinity, the solution is given by: For each  $k\in \{0,1,...,N-1\}$,
\begin{eqnarray*}
P_\lambda u(x) = x^{- 1/2 }A_{k}\ I_{1/2}(x\sqrt{2\lambda})+  x^{-1/2}B_{k}\ K_{1/2}(x\sqrt{2\lambda}),\ \mbox{in} \ [k,k+1],
\end{eqnarray*}
with $B_0=0$ and
\begin{eqnarray*}
P_\lambda u(x) =   x^{-1/2}B_{N}\ K_{1/2}(x\sqrt{2\lambda}),\ \mbox{in} \ [N,\infty).
\end{eqnarray*}
The constants $A_k,B_k$ has to be determined according to the second condition of the differential equation. Hence for $k=0,1,\cdots,N-1$ constants $A_k, B_k$ are given by Lemma \ref{Coefficients}, whereas $B_N=   \frac{N^{1/2} u(N)}{K_{1/2}(N\sqrt{2\lambda})}$.
\begin{theo}
\begin{enumerate}
\item $\ccalD=\R^N$ and for each $u=(u_1,\cdots,u_N)\in\R^N$ it holds
\begin{align*}
\ccalE[u]=\sum_{k=1}^{N-1} k(k+1)( u_{k+1} - u_k )^2+  N u(N)^2.
\end{align*}
\item $\ccalE$ is transient irreducible.
\item $\ccalE$ is not conservative.
\end{enumerate}
\end{theo}
\begin{proof}
Assertions 1. and 2. can be proved as for the case of infinite support, whereas assertion 3. follows from the fact $\ccalE[1]=N\neq 0$.

\end{proof}
\section{ The trace of Bessel's Dirichlet form with respect to measures of mixed type}
We consider once again the Dirichlet form  $\calE$ associated to the Bessel operator with $n=3$,  however we shall change the measure. We fix  a measure $\mu$ on $[0,\infty)$ of  mixed  type, i.e. a measure which has an absolutely continuous part and a discrete part. Precisely,
$$
\mu = \mu_{ac} + \mu_{disc},
$$
with
\begin{eqnarray}
\mu_{ac}= x^2\mathbf{1}_{(0,1)} dx,\ \mu_{disc}=\sum_{k=1}^{\infty}a_k \delta_k,\ a_k>0,\ \forall\,k\in\N.
\end{eqnarray}
Then the support $F$ of $\mu$ is  $F= [0,1]\cup\{ k,\ k\in \N \}$.\\
Let us compute the trace of $\calE$ w.r.t. the measure $\mu$. In this case the operator $J$ is defined as usual through
$$
J: \calD\cap L^2 (F,\mu) \to  L^2 (F,\mu),\ Ju= u_{|_{F}}.
$$
Then
\[
Ker(J):= \{ u\in\calD:\ u_{|_{(0,1)}}=0,\ u(k)=0, \ \forall k\in \N \},
\]
and for each $\lambda>0$, $P_\lambda u$ is the solution of
\begin{align}
    -\frac{1}{2}(P_\lambda u)'' -\frac{1}{2x}(P_\lambda u)' +\lambda P_\lambda u & = 0 \quad\text{ in } \ \bigcup_{k=1}^{\infty} (k,k+1),\nonumber\\
    P_\lambda u & = u \quad\text{ on } (0,1)\cup \N.
\end{align}
Thus for each integer $k$ the solution is given by
$$
 P_\lambda u(x) = x^{- 1/2 }A_{k}\ I_{1/2}(x\sqrt{2\lambda})+  x^{-1/2}B_{k}\ K_{1/2}(x\sqrt{2\lambda}),\ \mbox{in} \ [k,k+1],
$$
where $A_{k}$ and $B_{k}$ are two real constants to be determined.\\
We first us compute $\check{\calE_\lambda}$.
\begin{lem}
For each $\lambda>0$, it holds
\begin{eqnarray}
\check{\calE_\lambda}[Ju]&=& \int_{0}^{1}(u'(x))^{2}x^{2} dx+\lambda\int_{0}^{1} u^{2}(x) 2 x^{2}dx\nonumber\\
& +& \sum_{k\in \N} [(P_\lambda u)'(k+1) u((k+1)^-)(k+1)^{2}-(P_\lambda u)'(k^+) u(k)k^{2}]
\end{eqnarray}
\end{lem}
\begin{proof}
We have
\begin{align}
\check{\calE_\lambda}[Ju] &=  \calE_\lambda[P_\lambda u]\\
&= \int_{0}^{\infty}((P_\lambda u)'(x))^{2}x^{2}dx+ \lambda \int_{0}^{\infty}(P_\lambda u)^{2}(x) 2 x^{2}dx\nonumber\\
& =  \int_{0}^{1}((P_\lambda u)'(x))^{2}x^{2}\,dx+\int_0^1\lambda(P_\lambda u)^{2}(x) 2 x^{2}dx+\int_{1}^{\infty}((P_\lambda u)')^{2}(x) x^{2}dx\nonumber\\
& + \lambda \int_{1}^{\infty}(P_\lambda u)^{2}(x) 2 x^{2}dx\nonumber\\
&= \sum_{k\in \N} \int_{k}^{k+1}\big( -\frac{1}{2}(P_\lambda u)''(x)(P_\lambda u)(x) -\frac{1}{x}(P_\lambda u)'(x)(P_\lambda u)(x) + \lambda (P_\lambda u)^{2}(x)\big) 2x^{2}dx\nonumber\\
& +  \sum_{k\in \N} (P_\lambda u)'(x)(P_\lambda u)(x)\ x^{2}|_{k}^{k+1} +\int_{0}^{1}(u'(x))^{2}x^{2} dx+\lambda\int_{0}^{1} u^{2}(x) 2 x^{2}dx\nonumber\\
& =  \sum_{k\in \N} (P_\lambda u)'(x)(P_\lambda u)(x)\ x^{2}|_{k}^{k+1}+\int_{0}^{1}(u'(x))^{2}x^{2} dx+\lambda\int_{0}^{1} u^{2}(x) 2 x^{2}dx.\nonumber\\
\end{align}
Finally we get
\begin{eqnarray}
\check{\calE_\lambda}[Ju]&=& \int_{0}^{1}(u'(x))^{2}x^{2} dx+\lambda\int_{0}^{1} u^{2}(x) 2 x^{2}dx\nonumber\\
& +& \sum_{k\in \N} [(P_\lambda u)'((k+1)^-) u(k+1)(k+1)^{2}-(P_\lambda u)'(k^+) u(k)k^{2}]
\end{eqnarray}
\end{proof}
Letting $\lambda \downarrow 0$, we get by using monotone convergence theorem for series
\begin{eqnarray*}
\check{\calE_0}[J u]
& = & \int_{0}^{1}u'(x)x^2 dx +  \sum_{k=1}^\infty k(k+1)( u(k+1) - u(k) )^2 \\
& = & \int_{0}^1(u'(x))^{2} d\mu_{a.c} + \sum_{k=1}^\infty k(k+1)( u(k+1) - u(k) )^2.
\end{eqnarray*}
An elementary computation shows that the latter form is closable.  Regarding the construction of $\ccalE$, we get $\ccalE=\overline{\calE_0}$. In order to obtain a precise description for $\ccalE$ we introduce the space
\begin{align}
\ccalD_{max}&=\{ \psi\in L^2(F,\mu):\, \psi\in AC[0,1],\\
&\int_{0}^1(\psi'(x))^{2}\,d\mu_{ac}+\sum_{k=1}^\infty k(k+1)( \psi(k+1) - \psi(k) )^2<\infty\}
\end{align}
and the quadratic forms $\check{\calE}^{(c)}, \check{\calE}^{(J)}$:
\begin{align*}
\dom \check{\calE}^{(c)}&=\dom\check{\calE}_0^{(J)}=\ccalD_{max},\\ &\check{\calE}^{(c)}[\psi]=\int_{0}^1(\psi'(x))^{2}\,d\mu_{a.c},\
\check{\calE}^{(J)}[\psi]= \sum_{k=1}^\infty k(k+1)( \psi(k+1) - \psi(k) )^2,\ \forall\,\psi\in \ccalD_{max}.
\end{align*}
Finally let
\begin{align*}
\mathcal{Q}:=\dom\mathcal{Q}=\ccalD_{max},\ \mathcal{Q}[\psi]= \check{\calE}^{(c)}[\psi]+\check{\calE}^{(J)}[\psi],\ \forall\,\psi\in\ccalD_{max}.
\end{align*}
\begin{lem}
The quadratic form $\mathcal{Q}$ is closed.
\label{ClosedSum}
\end{lem}
\begin{proof}
In fact, $\mathcal{Q}$ is the sum of two closed quadratic forms. From the part of the former section we already know that $\check{\calE}^{(J)}$ is closed.\\
Let us show that $\check{\calE}^{(c)}$ is closed. Let $(u_n)\subset\ccalD_{max}$ such that  $ \|u_n - u \|_{L^2 (F,\mu)} \longrightarrow 0$ as $n \to \infty$ and $ \check{\calE}^{(c)}[u_n - u_m]\longrightarrow 0$ as $ n,m\to \infty$. Then $(u'_n)$ is Cauchy sequence in $L^2 ((0,1),\mu_{a.c})$.\\
It is well known that $L^2 ((0,1),\mu_{a.c})$ is a Hilbert space so there exist $v\in L^2 ((0,1),\mu_{a.c})$ \ s.t \ $u'_n \to v $ in \ $L^2 ((0,1),\mu_{a.c})$. Since $ u_n \to u $ in  $L^2 (F,\mu)$ and hence in $L^2((0,1),\mu_{a.c})$, we obtain  $u'_n\to u' $ in the sense of  distribution. Hence  $u' = v$ and  $u\in\ccalD_{max}$, moreover $\check{\calE}^{(c)}[u_n - u]\to 0$.
\end{proof}

\begin{theo}
\begin{enumerate}
\item It holds $\ccalE=\overline{\mathcal{Q}|_{\ran J}}$.
\item If $\mu_{\disc}(\N)=\infty$, then  $\ccalE= \mathcal{Q}$.
\end{enumerate}
\label{TraceMixed}
\end{theo}
\begin{proof}
Assertion 1. follows from Lemma \ref{ClosedSum} together with the fact that $\ccalE_0=Q|_{\ran J}$.\\
2.: It suffices to prove that $\ran J$ is a core for $\mathcal{Q}$.\\
On the one hand, since $\mu_{disc}(\N)=\infty$, we know that $\ran J\cap\ell^2(\mu_{disc})$ is a core for $\ccalE^{(J)}$. On the other one, we have $C_c^1[0,1]\subset\ran J$ is a core for $\ccalE^{(c)}+2\delta_1$ in $L^2([0,1],\mu_{ac})$ which is the trace of $\calE$ w.r.t the measure $\mu_{ac}$. All these considerations lead to the fact that $\ran J$ is a core for $\mathcal{Q}$.
\end{proof}
We quote that $\calE$ is the sum of a Dirichlet form of strongly local type, $\check{\calE}^{(c)}$ and an other one of non-local type, $\check{\calE}^{(J)}$.
\begin{theo}
\begin{enumerate}
 The Dirichlet form $\ccalE$ is transient irreducible.
 \label{Trans2}
\end{enumerate}
\end{theo}
\begin{proof}
{\em Irreducibility}:  Assume that $\ccalE$ is not irreducible. Then, there is an invariant set $\emptyset\neq A \subsetneq F$ s.t $A= X\cup Y$ where $X=\subset (0,1)$ and $\emptyset\neq Y\subset\N$. If $Y=\subsetneq\N$, arguing as in the proof of Theorem 3.6 leads to a contradiction. Thus either $Y$ is empty or $Y=\N$. Suppose $Y=\N$. The function $u=\textbf{1}_{[0,1]}\in \ran J$. Thus we should have $u\textbf{1}_{A},\ \textbf{1}_{A^c}$ are in $\ccalD_{max}$ and $\calE[u]=\calE[u\textbf{1}_{A}] + u\textbf{1}_{A^c}]$. However, since $X\subset(0,1)$, we get $\calE[u\textbf{1}_{A}]=2=\calE[u\textbf{1}_{A^c}]=\calE[u]$ and we are led to a contradiction. Hence $Y=\emptyset$. Finally we achieve $\emptyset\neq A\subset (0,1)$. But then we would get $\textbf{1}_{(0,1)}\in\dom\ccalE$, which is absurd, because the latter function does not have any continuous representative. In any case we get a contradiction and then $\ccalE$ is irreducible.\\
{\em Transience}:  Let $u\in\calD$. Using inequality (\ref{Strauss}) we obtain
\begin{align}
(P_\lambda u (x))^2= (u(x))^2\leq \frac{c'}{x}\calE_\lambda[P_\lambda u],\ \forall\,x\in(0,1).
\label{PointG}
\end{align}
 Let $g:F\to (0,\infty),\ g = {e^{-x}}.\textbf{1}_{(0,1)}+ \frac{e^{-k}}{a_k}. \textbf{1}_{\N}$. Then $g>0$ and $g\in L^1 (F,\mu)$.  Using inequality (\ref{PointK}) and (\ref{PointG}) we get
\begin{align}
\int_{F} |P_\lambda u|g\,d\mu&=\int_0^1 |P_\lambda u(x)|g(x)\,d\mu_{a.c} + \int_{\N} |P_\lambda u(k)|g(k)\,d\mu_{disc}\nonumber\\
&=  \int_0^1 |\psi(x)|g(x)\ x^2 dx +  \sum_{k\geq 1} |u(k)|g(k)a_k \nonumber\\
&\leq C \sqrt{\calE_\lambda[P_\lambda u]} = C \sqrt{\ccalE_\lambda[Ju]},\ \forall\,u\in\calD.
\end{align}
Thus letting $\lambda\downarrow 0$ and taking into account that $\ran J$ is a form core for $\ccalE$ we achieve the inequality
\begin{align}
\int_{F} |u|g\,d\mu\leq c\sqrt{\ccalE[u]},\ \forall\,u\in\ccalD_{max}.
\end{align}
Therefore $\ccalE$ is transient.
\end{proof}
\noindent
To discuss conservation property for the obtained trace form, we shall apply Masammune--Uemura--Wang result, which asserts the following in the abstract frame of metric measure energy space see (\cite[Theorem 1.1]{MasamuneConserv}). Assume that conditions  \textbf{(A1), (A2), (A3)} and \textbf{ (M1), (M2)} (of \cite[Theorem 1.1]{MasamuneConserv}) are fulfilled.  Let $d$ be the metric of $F$. Suppose there is a finite constant $c>0$ such that
\begin{eqnarray}
 \sup_{x\in F} \int_{x\neq y} (1\wedge d(x,y)^2 ) J(x,dy ) \leq c.
 \label{jump-metric}
\end{eqnarray}
Then the condition
\begin{eqnarray}
\liminf_{k\to\infty} \frac{\ln V(k)}{k\ln k}<\infty
\label{volume-conservation}
\end{eqnarray}
yields the conservativeness of the Dirichlet form $\calE^{(c)} +\calE^{(J)}$. Here $V(k)$ is the volume of the ball of radius $k$.\\
For our situation it is elementary to prove that conditions  \textbf{(A1), (A2), (A3)} and \textbf{ (M1), (M2)}. Moreover the jump part of $\ccalE$ is
\[
\check{\calE}^{(J)}[\psi] = \sum_{k\in \N} k(k+1) (\psi(k+1) - \psi(k))^2
\]
We rewrite it as
$$
\check{\calE}^{(J)}[\psi] = \int_{  \N}\int_{  \N}  (\psi(k+1) - \psi(k))^2 \ \frac{b(k,j)}{a_k a_j}\ d\mu_{disc}(j) d\mu_{disc}(k),
$$
and the associated jump kernel is
\begin{eqnarray}
J(k, d\mu_{disc}(j)) = \frac{b(k,j)}{a_k a_j}\ d\mu_{disc}(j).
\label{jump}
\end{eqnarray}
We shall consider two metrics on $F$, first the Euclidean metric.
\begin{theo}
\begin{enumerate}
\item Assume that $\mu_{disc}(F)<\infty$. Then whatever the metric considered in $F$,  $\ccalE$ is not conservative.
\item  Assume that $\mu_{disc}(F)=\infty$.  We consider $F$ endowed with the Euclidean metric. If
\begin{eqnarray}
\sup_{k\in\N} \frac{k^2}{a_k}<\infty,
\label{CondMixedA}
\end{eqnarray}
and
\begin{eqnarray}
\liminf_{k\to \infty} \frac{\ln (\sum_{j=1}^{k} a_j)}{k \ln k} < \infty,
\end{eqnarray}
then $\ccalE$ is conservative.
\end{enumerate}
\label{ConsMixedA}
\end{theo}
\begin{proof}
Assertion 1. can be proved as the first assertion Theorem \ref{BesselCons}. We omit its proof.\\
2.  Condition (\ref{jump-metric}) reads
$$
\sup_{k\in \N} \frac{1}{a_k} \sum_{j: j\neq k} b(k,j) (1\wedge d(k,j)^2 )<\infty.
$$
A straightforward computation yields
\begin{eqnarray*}
\sup_{k\in \N} \frac{1}{a_k} \sum_{j: j\neq k} b(k,j) (1\wedge |k-j|^2 )
& = &  \sup_{k\in \N} \frac{1}{a_k} \left( b(k,k-1) + b(k,k+1) \right)\\
& = &  \sup_{k\in \N} \frac{1}{a_k} \left( \frac{k(k-1)}{2} + \frac{k(k+1)}{2} \right)  \\
& = &  \sup_{k\in \N} \frac{k^2}{a_k}.
\end{eqnarray*}
We have $V(k)= \frac{1}{3}+ \sum_{j=1}^k a_j$. Thus, for large $k$ the volume $V(k)$ is comparable to $\sum_{j=1}^k a_j$. Finally we  get the result owing to  \cite[Theorem 1.1]{MasamuneConserv}.
\end{proof}
We shall now consider an other metric on $F$ as in \cite[Example 5.4]{MasamuneConserv}. For vertices $k,j\in\N$ such that $ k\sim j$, we set
$$
\sigma(k,j) := \frac{1}{\sqrt{deg(k)}} \wedge \frac{1}{\sqrt{deg(j)}}\wedge 1,
$$
where
$$
deg(k) = \frac{1}{a_k} \sum_{j\sim k}\ b(k,j)= \frac{k^2}{a_k}.
$$
Thus
$$
\sigma(k,j) = \frac{\sqrt{a_k}}{k} \wedge  \frac{\sqrt{a_j}}{j}\wedge 1.
$$
We define the standard adapted metric $d$, first on the discrete part of $F$  by
$$
d(k,j) = \inf \{ \sum_{i=k}^{j-1} \sigma(i,i+1)\}\ \quad\text{ for }  k\leq j.
$$
For $k\geq j$, $d(k,j)$ is defined in an obvious way. The metric $d$ is then extended by linear interpolation  on the set $F$ to a metric which we still denote by $d$.\\
Let us observe that the metric $d$ is adapted, i.e.
\[
\sup_{k,j}\{d(k,j),\ k\sim j\}<\infty.
\]
\begin{theo}
Assume that $\mu_{disc}(F)=\infty$ and
\begin{eqnarray}
\liminf_{k\to \infty} \frac{\ln (\sum_{j=1}^{k} a_j)}{k \ln k} < \infty.
\end{eqnarray}
Then $\check{\calE}$ is conservative.
\label{ConsMixedB}
\end{theo}
\begin{proof}
It suffices to prove
\[
\sup_{k\in \N} \frac{1}{a_k} \sum_{j: j\sim k} b(k,j) (1\wedge d(k,j)^2 )<\infty,
\]
and to use \cite[Theorem 1.1]{MasamuneConserv}.\\
A straightforward computation yields
\begin{eqnarray*}
\frac{1}{a_k} \sum_{j: j\sim k} b(k,j) (1\wedge d(k,j)^2 )
& = &  \frac{1}{a_k} \big( (1\wedge d(k,k-1)^2 )b(k,k-1)\\
& +& (1\wedge d(k,k+1)^2 )b(k,k+1) \big)\\
& = &  \frac{1}{a_k} \left(  \sigma(k,k-1)^2 \frac{k(k-1)}{2} + \sigma(k,k+1)^2 \frac{k(k+1)}{2} \right)  \\
&\leq & \frac{1}{a_k} \left( \frac{a_k}{k^2}\frac{k(k-1)}{2} +  \frac{a_k}{k^2}\frac{k(k+1)}{2}\right)
=1,
\end{eqnarray*}
which was to be proved.
\end{proof}
\begin{rk}
\begin{enumerate}
\item According to Theorem \ref{Conservative}-2), condition  (\ref{CondMixedA}), yields also the conservativeness of the trace of $\calE$ on $\N$.
\item Compared to Theorem \ref{ConsMixedA}, the latter theorem is stronger. Indeed, if we choose $\mu_{disc}$ to be the counting measure then condition (\ref{CondMixedA}) is not fulfilled. Accordingly we are  not able to conclude about the conservativeness of $\ccalE$. However, with the help of Theorem \ref{ConsMixedB} we obtain the conservativeness of $\ccalE$.

\end{enumerate}
\end{rk}

\bibliography{BiblioCons}

\end{document}